\documentclass[11pt, a4paper, twoside, reqno]{amsart}

\usepackage[english]{ babel}

\usepackage[utf8]{inputenc}
\usepackage[T1]{fontenc}
\usepackage{lmodern}
\usepackage{verbatim}

\usepackage{amsmath}
\usepackage{amssymb}
\usepackage{amsthm}

\usepackage{csquotes}

\usepackage{mathtools}
\usepackage{stmaryrd}

\usepackage{wrapfig}
\usepackage{subcaption}
\usepackage{multirow}

\usepackage{sidenotes}
\usepackage{enumitem}
\usepackage{verbatim}
\usepackage{wrapfig}
\usepackage{makecell}
\usepackage{xcolor}
\usepackage{todonotes}

\usepackage{hyperref}
\hypersetup{backref, pdfpagemode=FullScreen, colorlinks=true,
  citecolor=magenta, linkcolor=cyan, urlcolor=blue}

\usepackage{csquotes} % Load csquotes for robust quote handling

\usepackage[
    style=numeric-comp,     % Our base style
    maxnames=5,             % Authors before 'et al.'
    giveninits=true,        % M. Aigner
    uniquename=false,       % Avoid extra initials for disambiguation
    doi=true,               % Show DOI
    url=true,               % Show URL
    isbn=true,              % Show ISBNs
    eprint=true,            % Show eprint (e.g., arXiv)
    backend=biber           % Essential for BibLaTeX
]{biblatex}

\makeatletter
\def\blx@addquote#1{%
  #1%
}
\DeclareFieldFormat{title}{#1}
\DeclareFieldFormat{citetitle}{#1}
\DeclareFieldFormat{booktitle}{#1}
\DeclareFieldFormat{maintitle}{#1}
\makeatother

\renewbibmacro{in:}{}

\DeclareFieldFormat{pages}{#1}

\DeclareFieldFormat{journaltitle}{\mkbibemph{#1}}
\DeclareFieldFormat{series}{\mkbibemph{#1}}
\DeclareFieldFormat{booktitle}{\mkbibemph{#1}} % For booktitles in inproceedings/incollection

\usepackage{xurl} % Handles long URLs gracefully
\usepackage{hyperref} % For clickable links. This must be loaded AFTER biblatex.
\hypersetup{
    colorlinks=true,
    linkcolor=blue,
    citecolor=blue,
    urlcolor=blue,
    filecolor=magenta,
    pdfpagelayout=OneColumn,
    pdftitle={Your Paper Title},
    pdfauthor={Your Name},
    pdfsubject={Mathematics},
    pdfkeywords={keyword1, keyword2}
}

\DeclareFieldFormat{url}{\url{#1}}
\DeclareFieldFormat{doi}{\url{https://doi.org/#1}}
\DeclareFieldFormat{eprint:arxiv}{arXiv:\,\url{https://arxiv.org/abs/#1}} % As in sample paper

\usepackage{tikz}

\DeclareMathOperator*{\pos}{pos}
\DeclareMathOperator*{\lin}{lin}
\DeclareMathOperator*{\rg}{rg}

\DeclareMathOperator*{\conv}{conv}
\DeclareMathOperator*{\aff}{aff}
\DeclareMathOperator*{\bd}{bd}
\DeclareMathOperator*{\cl}{cl}

\DeclareMathOperator*{\inter}{int}

\newcommand{\R}{\mathbb{R}}

\newcommand{\N}{\mathbb{N}}

\mathtoolsset{showonlyrefs}

\newtheorem{theorem}{Theorem}[section] % reset theorem 
\newtheorem{corollary}[theorem]{Corollary}
\newtheorem{proposition}[theorem]{Proposition}
\newtheorem{lemma}[theorem]{Lemma}

\newtheorem{conjecture}[theorem]{Conjecture}
\newtheorem{remark}[theorem]{Remark}
\newtheorem{theoman}{Theorem}[section]

\theoremstyle{definition}

\theoremstyle{remark}

\newtheorem{example}[theorem]{Example}

\newcommand{\Sn}{\mathbb{S}^{n-1}}

\newcommand{\vol}{\mathrm{vol}\,}
\newcommand{\sph}{\mathbb{S}}% sphere
\newcommand{\scc}{{\tt scc}} 
\newcommand{\cv}{{\tt cv}}

\newcommand{\ov}{\overline}
\DeclareMathOperator{\V}{V}
\DeclareMathOperator{\relint}{relint}
\let\emptyset\varnothing

\numberwithin{equation}{section}
\begin{document}

\title{On polynomial inequalities for cone-volumes of polytopes}
\author{Tom Baumbach}
\author{Martin Henk}

\address{Technische Universität Berlin, Institut für Mathematik, Sekr.~MA4-1, Straße des 17.~Juni 136, 10623 Berlin, Germany}
\email{baumbach@math.tu-berlin.de}
\email{henk@math.tu-berlin.de}

       \begin{abstract}
         Motivated by the discrete logarithmic Minkowski problem we study for a
given matrix $U\in\mathbb{R}^{n\times m}$  its cone-volume set $C_{\tt cv}(U)$ 
consisting of all the cone-volume vectors of polytopes 
$P(U,b)=\{ x\in\mathbb{R}^n : U^\intercal  x\leq b\}$,
$b\in\mathbb{R}^n_{\geq 0}$. We will show that $C_{\tt cv}(U)$ is a
path-connected semialgebraic
set which extends former results in the planar
case or for particular polytopes. Moreover, we define 
a subspace concentration polytope $P_{\tt scc}(U)$ which represents geometrically  the subspace
concentration conditions for a finite discrete Borel measure
on the sphere.  This is up to a scaling the 
basis matroid polytope of $U$, and these two sets, $P_{\tt scc}(U)$ and
$C_{\tt cv}(U)$, also offer a new geometric 
point of view to the discrete logarithmic Minkowski
problem. 
       \end{abstract}

   \maketitle

	\section{Introduction}
The setting for this paper is the $n$-dimensional Euclidean space $\R^n$. For two vectors $x,y \in \R^n$ we denote by $\langle x,y \rangle$ the standard scalar product of $x$ and $y$, and  $\| x \| = \sqrt{\langle x,x \rangle}$  denotes the associated Euclidean norm; $\sph^{n-1}=\{x\in\R^n : \| x \|=1\}$ is the $(n-1)$-sphere.  The convex hull of a non-empty set $M\subset\R^n$ is denoted by $\conv M$, and if $M$ is finite then $\conv M$ is called a polytope.     By a result attributed to Minkowski and Weyl, $P\subset\R^n$ is a polytope if and only if 
    \begin{equation*}
       P=P(U,b)=\{x\in\R^n : U^\intercal x\leq b\}
    \end{equation*}
    for a  matrix $U=(u_1,\dots,u_m)\in(\Sn)^m$ with $\pos U=\R^n$ and $b\in\R^m$. 
    Here $\pos U$ means the positive hull, i.e., the set of all non-negative linear combinations of the column vectors $u_1,\dots,u_m\in\Sn$ of $U$.  Apparently, we may assume that the column vectors are pairwise different, and therefore we set
    \begin{equation*}
      \mathcal{U}(n,m)=\left \{U=(u_1,\dots,u_m)\in(\Sn)^m : \pos U=\R^n, u_i\ne u_j, i \ne j\right\}. 
    \end{equation*}

    For $1\leq i\leq m$ let
\begin{equation*} 
  F_i(b)=F(u_i,b)=P\cap\{x\in\R^n : \langle u_i,x\rangle =b_i\}
\end{equation*}
  which is always a face of $P$, and, of course, might be empty. If $\dim F_i(b)=\dim P-1$, $F_i(b)$ is called a facet of $P$. For $M\subset \R^n$ we denote by $\vol(M)$ its volume, i.e., its $n$-dimensional Lebesgue measure. If $M$ is contained in a
$k$-dimensional plane $A$, $\vol_k(M)$ refers to the $k$-dimensional Lebesgue measure with respect to $A$.

We will mainly assume that $b\geq 0$. This implies $0\in P$, and if $b > 0$ then  $0 \in \inter P$, i.e., $0$ is an interior point of $P$, and so $\dim P=n$.   
If  $F_i(b)$ is a facet of $P(U,b)$, $\dim P(U,b)=n$, 
then $\frac{1}{n}b_i\vol_{n-1}(F_i(b))$ is the volume of the cone (pyramid)
$\conv(\{0\}\cup F_i(b))$. For $U\in\mathcal{U}(n,m)$, the polytope $P(U,b)$ is the interior-disjoint union of all these cones, so we can write
\begin{equation*}
            \vol(P(U,b))=\frac{1}{n}\sum_{i=1}^m b_i\vol_{n-1}(F_i(b)).
          \end{equation*}
For such a $P=P(U,b)$, $\dim P=n$, we consider its cone-volume measure $\V_P$ which is  
the finite non-negative Borel measure $\V_P:\mathcal{B}(\sph^{n-1})\to\R_{\geq 0}$          
 given by
\begin{equation*}
  \V_P(\eta) =\sum_{i=1}^m \frac{b_i}{n}\vol_{n-1}(F_i(b))\,\delta_{u_i}(\eta) =\sum_{u_i\in\eta} \frac{b_i}{n}\vol_{n-1}(F_i(b)). 
\end{equation*}
Here $\eta\subseteq \sph^{n-1}$ is a Borel set and  $\delta_{u_i}(\cdot)$ is the Dirac measure in $u_i$, i.e., $\delta_{u_i}(\eta)=1$ if $u_i\in\eta$, otherwise it is $0$.  

The discrete logarithmic Minkowski (existence) problem introduced by Böröczky, Lutwak, Yang and Zhang \cite{BoeroeczkyLutwakYangEtAl2025} asks for necessary and sufficient conditions such that a finite discrete Borel measure  
\begin{equation}
  \mu:\mathcal{B}(\sph^{n-1})\to\R_{\geq 0}, \quad \mu(\eta)=\sum_{i=1}^m \gamma_i\,\delta_{u_i}(\eta)
\label{eq:measure}  
\end{equation}
with $u_i\in\sph^{n-1}$, $\gamma_i> 0$, is the cone-volume measure of a polytope.  We will denote such a measure also by  $\mu(U,\gamma)$, where $\gamma\in\R^m_{>0}$ is the vector with entries $\gamma_i$.

This discrete problem can be extended to the continuous setting, i.e., to the space of all convex bodies and  
the corresponding general logarithmic Minkowski problem is a cornerstone of modern convex geometry. 
The associated partial differential equation for the logarithmic Minkowski problem is the following Monge-Ampère type equation on the unit sphere: For a given function $f: \sph^{n-1} \rightarrow (0,\infty)$, solve for the support function $h: \sph^{n-1} \rightarrow (0,\infty)$ of a convex body,
\begin{equation*}
    h \det(h_{ij} + h \delta_{ij}) = f,
\end{equation*}
where $h_{ij}$ is the covariant derivative of $h$ with respect to an orthonormal frame on $\sph^{n-1}$ and $\delta_{ij}$ is the Kronecker delta. The Monge-Ampère equation has a close relation to the optimal transport with quadratic cost \cite{de2014monge}. 

The cone-volumes are instrumental for computing Wachspress coordinates \cite{ju2005geometric}, which define the adjoint of a polytope and thereby its canonical form. The canonical form was introduced in the context of positive geometries and scattering amplitudes in quantum field theory \cite{arkani2017positive}. The adjoint of a polytope appears in many different mathematical contexts (cf.~\cite{kohn2020projective}) and is relevant for convex optimization \cite{Dmitrii2025}.

The cone-volume measure has found numerous important applications in convex geometry and analysis.
In particular, its properties have been used to establish reverse affine isoperimetric inequalities 
\cite{HE200673,XIONG20103214}. Since the work of Gromov and Milman 
\cite{CM_1987__62_3_263_0}, it has become a central tool, with further 
applications to functional inequalities, asymptotic geometric analysis, and probability theory 
\cite{accb6532790c46f8b15a8d2da8ff473f,Grigorios2009,Barthe2005}.
For its history, relevance and impact we refer to \cite{HuangYangZhang2025, BoeroeczkyLutwakYangEtAl2025, Stancu2016, BoeroeczkyHenk2016a, LiuSunXiong2021, Zhu2014, Boeroeczky2023, BoeroeczkyHegedusZhu2016, ChenLiZhu2019} and to the references within. 
Here we will only focus on the discrete setting. 

The subspace concentration condition (\scc), introduced by Böröczky et al.~\cite{BoeroeczkyLutwakYangEtAl2025}, plays an important role in the classification of the cone-volume measure.   A finite  discrete Borel measure $\mu=\mu(U,\gamma): \mathcal{B}(\sph^{n-1})\to\R_{\geq 0}$ with $U\in\mathcal{U}(n,m)$, $\gamma >0$, is said to satisfy the \scc\, if
\begin{enumerate}
\item for every proper linear subspace $L\subset\R^n$ it holds 
\begin{equation}
           \mu(L)=  \sum_{u_i\in L} \gamma_i \leq \frac{\dim L}{n}\mu(\sph^{n-1}),
 \label{eq:scc1} 
  \end{equation}
\item   and equality holds in \eqref{eq:scc1} if and only if there exists a subspace $\ov L$ complementary to  $L$ such that $\{u_1,\dots,u_m\}\subset L\cup\ov L$.   
\end{enumerate}   

In Section 2 we will define for $U\in\mathcal{U}(n,m)$ the polytope $P_{\scc}(U)$ (see \eqref{def:pscc}), which we call the subspace concentration polytope (of $U$) that captures the \scc. Up to scaling  the polytope $P_\scc(U)$ is (just) the matroid base polytope of the set of column vectors of $U$. 
\begin{proposition} Let $U\in\mathcal{U}(n,m)$ and $\gamma\in\R^m_{>0}$ with  $\sum_{i=1}^m \gamma_i=1$.  
  Then the finite discrete Borel measure $\mu(U,\gamma)$ satisfies the \scc\, if and only if $\gamma\in \relint P_\scc(U)$.
\label{prop:main}  
\end{proposition}
Here $\relint(M)$ denotes the relative interior of $M\subseteq\R^n$, i.e., the sets of interior points with respect to the ambient space
given by $\aff M$, the affine hull of $M$.

In the special case that $U$ does not contain parallel vectors, the polytope 
$P_\scc((U,-U))$ as well as Proposition 1.1 with the additional symmetry assumption $\gamma_i=\gamma_{m+i}$, $1\leq i\leq m$, was already considered by Liu et al.~\cite[Thm.~4.7]{LiuSunXiong2024b}.   

In order to show the relation of the \scc\, to the cone-volume measure we also define a cone-volume set $C_\cv(U)$. To this end, we firstly consider for $U\in\mathcal{U}(n,m)$ and  $b\in \R^m_{\geq 0}$ the cone-volume vector 
\begin{equation*}
  \gamma(U,b)=\frac{1}{n}\Big(b_1\vol_{n-1}(F(u_1)),\dots, b_m\vol_{n-1}(F(u_m))\Big)^\intercal\in \R^m_{\geq 0}. 
\end{equation*}
Observe that some of its entries might be zero, if $F_i(b)$ is not of dimension $n-1$ or $b_i=0$.  The set 
\begin{equation*}
  C_\cv(U)=\left\{ \gamma(U,b) : b\in\R^m_{\geq 0} \text{ and } \vol(P(U,b))=1\right\}
\end{equation*}
is called the cone-volume set of $U$.   
Any cone-volume vector $\gamma(U,b)$ of an $n$-dimensional polytope of the type $P(U,b)$ is up to scaling to volume 1 an element of $C_\cv(U)$ as   
\begin{equation}
 \frac{1}{\vol(P(U,b))}\gamma(U,b) = \gamma\left(U, \left(\vol(P(U,b)) \right)^{-1/n}b\right)\in C_\cv(U).
 \label{eq:scaling} 
\end{equation}

If $m=2m'$ is even and $u_{m'+i }=-u_i$, $i=1,\dots,m'$, we denote such a matrix by $U^s\in \mathcal{U}(n,m)$ and a vector $b\in\R^m$ satisfying $b_i=b_{m'+i}$, $i=1,\dots,m'$, will be denoted by $b^s$. Let 
\begin{equation*}
  C^s_\cv(U^s)=\left\{ \gamma(U^s,b^s) : b^s\in\R^m_{\geq 0}, \vol(P(U^s,b^s))=1\right\}
\end{equation*}
be the associated  symmetric cone-volume set. In the groundbreaking paper \cite{BoeroeczkyLutwakYangEtAl2025} it was in particular shown
that an even finite positive Borel measure $\mu(U^s,\gamma^s)$ is the cone-volume measure of an origin symmetric polytope $P(U^s,b^s)$ if and only if $\mu(U^s,\gamma^s)$ satisfies \scc. Hence, with 
Proposition \ref{prop:main} this can be reformulated as (see also \cite[Thm.~4.7]{LiuSunXiong2024b})
\begin{theoman}[\protect{\cite[][Thm.~1.1]{BoeroeczkyLutwakYangEtAl2025}}]  Let $m=2m'$ and $U^s\in\mathcal{U}(n,m)$. Then it holds 
  \begin{equation*}
    C^s_\cv(U^s) \cap\R^m_{>0}=\relint \left(P_\scc (U^s) \cap % \bigcap_{i = 1}^{m'} 
    \left\{ x \in \R^m : x_{i } = x_{m'+ i},\,1\leq i\leq m' \right\}\right).
 \end{equation*}    
\label{thm:blyz}  
\end{theoman}

In the general setting we will show that  $C_\cv(U)$ and $P_\scc(U)$ coincide only for parallelepipeds.
\begin{theorem} Let $U\in\mathcal{U}(n,m)$. Then 
  $C_\cv(U)=P_\scc(U)$ if and only if  $m=2n$ and up to renumbering we have
  $u_{n+i}=-u_i$, $1\leq i\leq n$.
\label{thm:equality}  
\end{theorem}
For the (non-symmetric) discrete logarithmic Minkowski problem we do not know necessary and sufficient conditions. 
By a result of Chen et al. \cite{ChenLiZhu2019}, however,  we have the following inclusion.
\begin{theoman}[\protect{\cite[][Thm.~1.1]{ChenLiZhu2019}}] Let $U\in\mathcal{U}(n,m)$. Then it holds
  \begin{equation*}
    C_\cv(U)\cap\R^m_{>0} \supseteq \relint P_\scc(U).
  \end{equation*}
  \label{thm:inclusion}
\end{theoman}

In Section 2 we will also see that both sets have the same dimension (Proposition \ref{prop:sccdirect} and Proposition \ref{prop:cvdirect}).

By definition, for $U\in\mathcal{U}(n,m)$, the cone-volume set $C_\cv(U)$ is a subset of $\{x\in\R^{m}: x\geq 0, \, x_1+x_2+\dots + x_m=1\}$. A result of Zhu shows that it can also be that large.
\begin{theoman}[\protect{\cite[][Thm.]{Zhu2014}}] Let  $U\in\mathcal{U}(n,m)$ be in general position, i.e., any $n$ columns of $U$ are linearly independent. Then 
  \begin{equation*}
    \begin{split} 
      C_\cv(U) \cap\R^m_{>0} &= \left\{x\in\R^{m}:  x>0,\, x_1+x_2+\dots + x_{m}=1\right\}\\ & = \conv\{e_1,\dots,e_m\}\cap\R^m_{>0} .
   \end{split}   
  \end{equation*}
\label{thm:zhu}  
\end{theoman}

\begin{remark} With some additional considerations and  a result of Zhu \cite[Theorem 4.3]{Zhu2014}, one can even show that for  $U\in\mathcal{U}(n,m)$ in general position it holds
  \begin{equation*}
     C_\cv(U) = \conv\{e_1,\dots,e_m\}.
   \end{equation*}
\end{remark}   

In general the inclusion in Theorem \ref{thm:inclusion} is strict as
$C_\cv(U)$ might not be convex and even not representable as the finite union of
polytopes (see Section 2). Our main result is that  $C_\cv(U)$ is (at least) a semialgebraic set, i.e., roughly speaking, it can be described by the finite union of sets which are representable by finitely many polynomial inequalities.

\begin{theorem} Let  $U\in\mathcal{U}(n,m)$. Then $C_\cv(U)$ is a semialgebraic set.
  \label{thm:main}  
\end{theorem}
In the special case  $n=2$, this was shown already by Stancu \cite{stancu2002discrete},
explicit descriptions of $C_\cv(U)$
for planar quadrilaterals were obtained by Liu et al.~\cite{LiuLuSunEtAl2024}, where the trapezoid case was already studied by Pollehn \cite[Section 2.4]{Subspace_Concentration_of_Geometric_Measures}.
In addition, a general valid polynomial inequality for arbitrary $C_\cv(U)$ was obtained by Böröczky and Heged{\H u}s \cite{BoeroeczkyHegedues2015}. 
Representations related to particular higher dimensional convex bodies were recently studied by Chen, Liu and Xiong (private communication). 

Our general polynomial description reduces to Stancus representation in the planar case. We will also present a bound on the degree of the polynomials in the general case (see Corollary \ref{cor:degree}).

The paper is organized as follows. In Section 2 we will define $P_{scc}(U)$, give a proof of Proposition \ref{prop:main},  show the relation to matroid polytopes, study certain basic properties of the two sets $P_\scc(U)$ and $C_\cv(U)$, and also provide a few examples. In particular, we will also show that $C_\cv(U)$ is path-connected (see Proposition \ref{prop:pathconnected}) and we will provide the proof of Theorem \ref{thm:equality}.  
The proof of Theorem \ref{thm:main} is given in Section 3 where we will also present some necessary background on semialgebraic sets. Section 4 deals with the 2-dimensional case, and in Section 5 we wil briefly discuss the non-uniqueness of cone-volume vectors.

\section{Subspace concentration polytopes and cone-volume sets}
\label{Section:ConeVolumeSet}

In the following let $U\in\mathcal{U}(n,m)$. With $S\subseteq U$ we mean a
subset of the column vectors, and $\rg(S)$ denotes the rang of the
matrix $S$, i.e., $\dim(\lin S)$. We will treat $S\subseteq U$ as
matrix as well as the set consisting of its  column vectors.

Let $\mathcal{B}(U)$ denotes all subsets of $U$
forming a basis of $\R^n$, then the tuple $M_U=(U, \mathcal{B}(U))$
is called the basis matroid of $U$. The associated characteristic polytope
\begin{equation*}
P(M_U) =\conv\left\{ \chi_U(B) : B\in \mathcal{B}(U) \right\}\subset\R^m
\end{equation*}
is called the (basis) matroid polytope of $M_U$.  Here $\chi_U(B)\in\R^m$ is the
characteristic vector of the basis $B$ with respect to $U$, i.e., for $1\leq
i\leq m$ its $i$th entry is $1$ if column $u_i\in B$, otherwise $0$.
For general information on matroids we refer to
\cite{GroetschelLovaszSchrijver1993, Aigner1997}. It is
well-known that $P(M_U)$ can also be described by the following system
of inequalities (see, e.g., \cite[Proposition 2.2]{FeichtnerSturmfels2005})
\begin{equation*}
  \begin{split}
    P(M_U) & = \left\{x\in\R^m : x\geq 0,\, \sum_{i=1}^m x_i=n, \sum_{u_i\in S}
    x_i\leq \rg(S) \text{ for all } S\in \mathcal{L}(U)\right\},
    \end{split}
  \end{equation*}
  where
  \begin{equation*}
  \begin{split}
  \mathcal{L}(U) & = \{ S \subseteq U : 1 \leq \rg(S) \leq n - 1
           \text{ and } U \cap \lin S = S \}.
    \end{split}
\end{equation*}
The subsets in $\mathcal{L}(U)$ are called flats, and for a flat $S$
the associated rank inequality is an (implicit) equality
for $P(M_U)$ if and only if $S$ belongs to the set
\begin{equation*}
    \mathcal{F}(U)  = \{ S \in \mathcal{L}(U) : \lin(S) \cap \lin(U \setminus S) = \{ 0 \} \},
  \end{equation*}
which are the so-called (non-trivial)
separators of the matroid
(see, e.g., \cite[pp. 315]{Aigner1997}, \cite{MR270945}).
  For later purpose and in view of \scc \,ii) we note that
  \begin{equation}
    \begin{split}
      \{\lin S :  S \in \mathcal{F}(U) \} = \{ & L : L\subset\R^n
      \text{ is a proper subspace such that } \\[-1ex]
                                   &\text{ there exits a
                                     complementary}\\[-1ex]
                                   &\text{ subspace }L'\text{ with }
                                   U\subset L\cup L' \}.
                                 \end{split}
    \label{eq:second}
  \end{equation}
With these two sets we define the subspace concentration polytope as
the base matroid polytope scaled by $1/n$: 
\begin{equation}
  \begin{split}
    P_\scc (U)= \frac{1}{n}  P(M_U) \\
              =\Biggl\{ x \in \R^m & : \sum_{i = 1}^m x_i = 1, \sum_{u_i \in S }  x_i = \frac{\rg(S)}{n},  S \in \mathcal{F}(U), \\
    &\, x\geq 0, \sum_{u_i \in  S }  x_i \leq \frac{\rg(S)}{n},\,  S \in
    \mathcal{L}(U) \setminus \mathcal{F}(U) \Biggl\}.
  \end{split}
  \label{def:pscc}
%\label{scc:vertices}  
\end{equation}

% We note that for any basis $B\in\mathcal{B}(U)$ and
% a separator $S\in\mathcal{F}(U)$ we have $|B\cap S|=\rg S$ and thus
% \begin{equation}
%   P(M_U) = n\,P_\scc(U).
% \end{equation}
Next we remark that $P_\scc(U)$ as well as $C_\cv(U)$ are
linear invariant which will be used later on.
\begin{proposition}
\label{prop:linearInvariant}Let $U\in\mathcal{U}(n,m)$ and let
  $A\in\R^{n\times n}$, $\det A\ne 0$.  Then
  $P_\scc(AU)=P_\scc(U)$ and $C_\cv(AU)=C_\cv(U)$.
\label{prop:invariant}
\end{proposition}
\begin{proof} The first identity follows from
  $A\mathcal{B}(U)=\mathcal{B}(AU)$ and \eqref{def:pscc}.
  For the second one we note that
$P(AU,b)=A^{-\intercal}P(U,b)$
and so
\begin{equation*}
  \gamma(AU,b)=|\det(A^{-\intercal})|\,\gamma(U,b)=\gamma(U,
  |\det(A^{-\intercal})|^{1/n}b).
\end{equation*}
Thus, $C_\cv(AU)=C_\cv(U)$.
\end{proof}

For the proof of Proposition \ref{prop:main} it will be convenient
first to give an explicit description of $\relint P_\scc(U)$. It
immediately follows from the above mentioned role of
the separators but for completness sake we add the short proof.

\begin{lemma} Let $U\in\mathcal{U}(n,m)$. Then
  \begin{equation*}
    \begin{split}
\relint P_\scc(U) =   \Biggl\{ x \in \R^m & : \sum_{i = 1}^m x_i = 1, \sum_{u_i \in S }  x_i = \frac{\rg(S)}{n},  S \in \mathcal{F}(U), \\
     &\, x > 0, \sum_{u_i \in  S }  x_i < \frac{\rg(S)}{n},  S \in
     \mathcal{L}(U) \setminus \mathcal{F}(U) \Biggl\}.
 \end{split}
\end{equation*}
\label{lem:relintscc}
\end{lemma}
\begin{proof} Apparently, the set on the right hand side is a subset
  of $\relint P_{\scc}(U)$. For the reverse inclusion let
  $y\in\relint P_\scc(U)$. Then $y$  admits a representation
  as (cf.~ \eqref{def:pscc},  \cite[Lemam 1.1.12]{Schneider2014})
  \begin{equation*}
    y=\sum_{B\in \mathcal{B}(U)} \lambda_B\,\frac{1}{n}\chi_U(B) \text{
      with }\lambda_B>0 \text{
      for all }  B\in \mathcal{B}(U)  \text{ and } \sum_{B\in \mathcal{B}(U)}\lambda_B=1.
  \end{equation*}
As each vector $u_i\in U$ is contained in some basis $B\in\mathcal{B}(U)$
we have $y>0$. Next, let $S \in
    \mathcal{L}(U) \setminus \mathcal{F}(U)$,
    $\rg(S)=k\in\{1,\dots,n-1\}$.   For each basis $B\in
    \mathcal{B}(U)$ we have
\begin{equation}
    \sum_{u_i\in S}\frac{1}{n}(\chi_U(B))_i =\frac{1}{n}|S\cap B|\leq
    \frac{\rg(S)}{n}
\label{eq:basisineq}
\end{equation}
    and so
    \begin{equation*}
       \sum_{u_i\in S}y_i = \sum_{u_i\in S} \sum_{B\in \mathcal{B}(U)}
       \lambda_B\,\frac{1}{n}(\chi_U(B))_i=  \sum_{B\in
         \mathcal{B}(U)}\lambda_B \sum_{u_i\in S}
       \frac{1}{n}(\chi_U(B))_i \leq \frac{\rg(S)}{n}.
    \end{equation*}
Hence, it suffices to show that there exists at least one basis $\ov
B$ with strict inequality in \eqref{eq:basisineq}: as $S\notin
\mathcal{F}(U)$ we have $\rg(U\setminus S) \geq n-\rg(S)+1$. Since
$S=\lin S\cap U$ 
% also  $\pos U=\R^n$ we conclude
% that 
we can find $n-\rg S +1$ linearly independent vectors $u_{j_i}\in U$
with $u_{j_i}\notin \lin S$. Supplementing these vectors to a basis
from $\mathcal{B}(U)$ gives a desired basis $\ov B$.
\end{proof}

Now we are ready to prove that $\relint P_\scc(U)$ describes the subspace
concentration conditions.
\begin{proof}[Proof of Proposition \ref{prop:main}]
First let us assume that the measure  $\mu(U,\gamma)$, $\gamma>0$,
satisfies the \scc. Then for
$S\in\mathcal{L}(U)$ we have by \scc\, i)
\begin{equation}
  \sum_{u_i\in S} \gamma_i =  \sum_{u_i\in \lin S}\gamma_i\leq
  \frac{\rg S}{n}.
\label{eq:proofprop1}
\end{equation}
Now by \eqref{eq:second} we have $S\in \mathcal{F}(U)$ if and only if
there exists a complementary subspace $L'$ to $L=\lin S$ with
$\{u_1,\dots,u_m\}\subset L\cup L'$. By \scc\,ii) this is equivalent to
having equality in \eqref{eq:proofprop1}.
% In the case $S\in \mathcal{F}(U)$, $\lin(U\setminus S)$ is a
% complementary subspace to $\lin S$ and so by \scc\, ii) we have equality
% in \eqref{eq:proofprop1}.
% If $S\notin \mathcal{F}(U)$ but if we would have equality in
% \eqref{eq:proofprop1} then by  \scc ii) $\lin(U\setminus L)$ has to be
% a complementary subspace and thus $\lin L\cap\lin(U\setminus
% L)=\{0\}$. Hence, we arrive at the contradiction $S\in
% \mathcal{F}(U)$ and so we have strict inequality in \eqref{eq:proofprop1}.
% for $S\notin \mathcal{F}(U)$.
In view of Lemma \ref{lem:relintscc} we conclude $\gamma\in\relint  P_{\scc}(U)$.

Let now $\gamma\in\relint  P_{\scc}(U)$. Then $\gamma>0$ and let
$L\subset\R^n$ be a proper subspace. With $S_L=U\cap L\in
\mathcal{L}(U)$ we have by Lemma \ref{lem:relintscc}
\begin{equation}
  \sum_{u_i\in L} \gamma_i=  \sum_{u_i\in S_L} \gamma _i\leq
  \frac{\rg (S_L)}{n} \leq \frac{\dim(L)}{n},
\label{eq:proofprop2}
\end{equation}
which shows \scc\, i). Moreover, we have equality in $\sum_{u_i\in L}
\gamma_i\leq \dim(L)/n$  if and only if
$S_L\in\mathcal{F}(U)$ and $\dim S_L=\dim L$  which by \eqref{eq:second} is equivalent to
the existences of a subspace $L'$ complementary to $\lin(S_L)=L$ (cf.~\eqref{eq:proofprop2}) with   $\{u_1,\dots,u_m\}\subset L\cup L'$. Thus \scc\,ii) is
verified as well.
\end{proof}

Before we proceed we have to extend the definitions of
$P(U,b)$, $C_\cv(U)$ and $P_\scc(U)$ to subsets $S\subseteq U$
with $\pos S=\lin S$. We will do this always with respect to the
``ambient matrix'' $U$, i.e., for a vector $v\in\R^{|U|}$ we denote by
$v_S\in\R^{|S|}$ the subvector of $v$ having coordinates $v_i$ with $u_i\in S$.
Then with $\mathcal{B}(S)=\{T\subseteq S: T \text{ basis of }\lin S\}$
we set
\begin{equation*}
  \begin{split}
    P_\scc(S) &=\conv\{\chi_U(T) : T\in \mathcal{B}(S) \}.
  \end{split}
\end{equation*}
With the canonical definitions of $\mathcal{L}(S)$, $\mathcal{F}(S)$
we have
\begin{equation*}
  \begin{split}
    P_\scc(S) & =\Bigg \{ x \in\R^{|U|} : x_{U\setminus S}=0, x_S\geq 0, \sum_{u_i\in
      S}x_i=1,\\
    &\quad\quad\quad  \quad\quad\quad\sum_{u_i\in T} x_i\leq \frac{\rg T}{\rg S} \text{ for }
    T\in\mathcal{L}(S),\\
     &\quad\quad\quad \quad\quad\quad \sum_{u_i\in T} x_i= \frac{\rg T}{\rg S} \text{ for }
    T\in\mathcal{F}(S) \Bigg \}.
  \end{split}
\end{equation*}
Regarding cone-volume sets let
\begin{equation*}
   P(S,b_S)=\{x\in\lin S: S^\intercal x\leq b_s\}.
 \end{equation*}
Let now $\gamma=\gamma(S,b_s)\in\R^{|U|}$  be the cone-volume vector
with $\gamma_{U\setminus S}=0$ and for $u_i\in S$ let $\gamma_i$ be
the associated cone-volume of the $\rg(S)$-dimensional polytope $P(S,b_s)$,
i.e.,
\begin{equation*}
        \gamma_i=\frac{b_i}{\rg S}\vol_{\rg(S)-1}(F_S(u_i)),
\end{equation*}
where $F_s(u_i)=P(S,b_s)\cap\{x\in\lin S : \langle u_i,x\rangle
=b_i\}$. Then we set
\begin{equation*}
  C_\cv(S)=\{\gamma(S,b_s) : b\in\R^{|U|}_{\geq 0} \text{ and }\vol_{\rg S}(P(S,b_s))=1\}.
\end{equation*}

Observe that for $S\in\mathcal{F}(U)$ we always have $\pos S=\lin S$, and with the separators we can write
$P_\scc(U)$ as direct sum of submatroid polytopes.
 In fact, given $S\in\mathcal{F}(U)$ it is
known (e.g., \cite[pp. 315]{Aigner1997}) that
\begin{equation}
                 P_\scc(U)=\frac{\rg(S)}{n}P_\scc(S)\oplus
                 \frac{\rg(U\setminus S)}{n}P_\scc(U\setminus S).
\label{eq:sepsplit}
\end{equation}
Iterating this process, i.e, looking at separators of $S$ and
$U\setminus S$ and so forth leads to a unique partion
\begin{equation}
  U=S_1\cup S_2 \cup \cdots \cup S_d
\label{eq:uniquepart}
\end{equation}
into so-called irreducible sets $S_j\subseteq U$, i.e.,
$\mathcal{F}(S_j)=\emptyset$, $1\leq j\leq d$.  In particular, we have
\begin{equation*}
    \R^n =\lin S_1\oplus \dots \oplus \lin S_d.
\end{equation*}

\begin{proposition} Let $U\in\mathcal{U}(n,m)$ and let $U=S_1\cup S_2 \cup
  \cdots \cup S_d$ be the unique partition into irreducible sets. Then
\begin{equation*}
   P_\scc(U)=\frac{\rg(S_1)}{n}P_\scc(S_1)\oplus  \cdots \oplus \frac{\rg(S_d)}{n}P_\scc(S_d),
 \end{equation*}
and $\dim P_\scc(U)= m-d$.
\label{prop:sccdirect}
\end{proposition}
\begin{proof} The decomposition follows from repeated application of
  \eqref{eq:sepsplit}. For the dimension see \cite[Prop.~
  2.4]{FeichtnerSturmfels2005}, or just observe that
  if $\mathcal{F}(S_j)=\emptyset$ then Lemma \ref{lem:relintscc} implies
  that $\dim P_\scc(S_j)=|S_j|-1$. Together with i) we get $\dim(P_\scc(U))=
  \dim P_\scc(S_1)+\dots + \dim P_\scc(S_d)=|S_1|+\dots+|S_d|-d=m-d$.
\end{proof}
% Subsets $S\subseteq U$ with $\mathcal{F}(U)=\emptyset$ are called
% irreducible.
Next we present three examples.
\begin{example}[Polytopes in general positions, e.g.,  a simplex]
  Let $U\in\mathcal{U}(n,m)$ be in general positions, i.e., each $n$ of the
  column vectors are linearly
  independent.  Then $\mathcal{F}(U)=\emptyset$ and for
  $S\in\mathcal{L}(U)$ we have $\rg S= |S|$.  Thus all inequalities
  for $S\in \mathcal{L}(U)$ are dominated by those with $|S|=1$. Hence
  \begin{equation*}
    \begin{split}
     P_\scc(U)& =\left\{x\in\R^{m}: x_1+\dots +x_{m}=1, x\geq 0, x_i\leq
     \frac{1}{n}, 1\leq i\leq m \right\}\\
      & = \frac{1}{n} \Big( [0,1]^{m}\cap\{x\in\R^{m}: x_1+\dots
        +x_{m}=n\} \Big) \\
      &=\frac{1}{n}\conv\left\{\sum_{i\in I} e_i: I\subset\{1,\dots,m\}, |I|=n\right\}.
   \end{split}
 \end{equation*}
% where $\widehat e_j$  means that the $j$th unit vector $e_j$ is
% omitted in the sum.
So, up to the factor $1/n$, $P_\scc(U)$ is the
hypersimplex $\Delta(n,m)$. \hfill $\triangle$
\label{ex:general}
\end{example}

\begin{example}[Parallelepiped] Let $u_1,\dots,u_n\in \sph^{n-1}$ be
  linearly independent and so $U^s=(u_1,\dots,u_{n},-u_1,\dots,-u_{n})\in\mathcal{U}(n,2n)$.
Then $\mathcal{F}(U)=\mathcal{L}(U) =\{(W,-W) : W\subset
(u_1,\dots,u_n), W\ne\emptyset\}$ and again, all equations resulting
from $\mathcal{F}(U)$ are dominated by those with $|W|=1$.
Hence,
  \begin{equation*}
    \begin{split}
     P_\scc(U^s)& =\{x\in\R^{2n}: x_1+\dots +x_{2n}=1, x\geq 0, x_i+x_{n+i}=
     \frac{1}{n}, 1\leq i\leq n \} \\
      &=\frac{1}{n}\Big(          \conv\{e_1,e_{n+1}\} \oplus\dots\oplus  \conv\{e_n,e_{2n}\}\Big) ,
   \end{split}
 \end{equation*}
 where the direct sum corresponds to the partition of $U^s$ into
the  irreducible sets  $S_j= \{u_j,u_{n+j}\}=\{u_j,-u_j\}$ (cf.~Proposition \ref{prop:sccdirect}). In particular, $P_\scc(U)$ is
a cube of dimension $n$ and of edge length $\sqrt{2}$.
\label{ex:parallel} \hfill $\triangle$
\end{example}

 \begin{example}[Trapezoid] Let $U=(u_1,u_2,u_3,u_4)\in \mathcal{U}(2,4)$,
   with  $u_3=-u_1$, and $\langle u_1,u_2\rangle, \langle
   u_1,u_4\rangle >0$. Then  $\mathcal{F}(U)=\emptyset$,
   $\mathcal{L}(U)=\{(u_1,u_3), (u_2), (u_4)\}$  and
   \begin{equation*}
      \begin{split}
     P_\scc(U)& =\left\{x\in\R^{4}: x_1+\dots +x_{4}=1, x\geq 0,
     x_2,x_4\leq\frac{1}{2}, x_1+x_3\leq \frac{1}{2}\right\} \\
      &=\conv\left\{(1,1,0,0)^\intercal, (1,0,0,1)^\intercal, (0,1,1,0) ^\intercal,
      (0,1,0,1) ^\intercal, (0,0,1,1)^\intercal\right\}.
   \end{split}
 \end{equation*}
Observe, that out of the 6 possible bases among 4 vectors only
$u_1,u_3$ do not build a basis. It is $\dim P_\scc(U)=3$ and $P_\scc(U)$ is
a pyramid over a square with apex $(0,0,1,1)^\intercal$. \hfill $\triangle$
\label{ex:trapez}
% \begin{figure}
%   \includegraphics[height=3cm]{screenshot-2.png.pdf}
% \end{figure}
 \end{example}

We further remark that the vertices of $P_\scc(U)$ are exactly the
vectors $\chi_U(B)$, $B\in\mathcal{B}$, and all
edges are parallel to a vector of the type $e_i-e_j$. \cite{FeichtnerSturmfels2005}

As well-understood as $P_\scc(U)$ is, the cone-volume set $C_\cv(U)$ is equally mysterious is (in general). But first we note that we also have a decomposition as in
Proposition \ref{prop:sccdirect}.
\begin{proposition} Let $U\in\mathcal{U}(n,m)$ and let $U=S_1\cup S_2 \cup
  \cdots \cup S_d$ be the unique partition into irreducible sets. Then it holds
\begin{equation}
   C_\cv(U)=\frac{\rg(S_1)}{n}C_\cv(S_1)\oplus \cdots \oplus\frac{\rg(S_d)}{n}C_\cv(S_d),
\label{eq:cvdirect}
 \end{equation}
 and $\dim(C_\cv(U))=m-d$.
\label{prop:cvdirect}
\end{proposition}
\begin{proof} By Theorem \ref{thm:inclusion} we know
  \begin{equation*}
              \relint P_\scc(U)\subseteq C_\cv(U)\subset \{x\in\R^m :
              x_1+\cdots +x_m=1\},
  \end{equation*}
where the right hand side inclusion follows immediately from the
definition of $C_\cv(U)$. Hence, $\dim(P_\scc(U))\leq
\dim(C_\cv(U))\leq |U|-1$ and if $U$ is irreducible, Proposition
\ref{prop:sccdirect} ii) yields
\begin{equation*}
                \dim(C_\cv(U))=|U|-1.
\end{equation*}
Thus, once we have established \eqref{eq:cvdirect} we obtain
$\dim(C_\cv(U)=m-d$ as in the proof of Proposition \ref{prop:sccdirect}.
In order to show \eqref{eq:cvdirect} let $S\in\mathcal{F}(U)$.
  By \eqref{eq:second} we know that $L=\lin(S)$ and
  $L'=\lin(U\setminus S)$ are complementary subspaces. It suffices to show
  that (see \eqref{eq:uniquepart})
  \begin{equation}
       C_\cv(U)= \frac{k}{n} C_\cv(S)\oplus \frac{n-k}{n}
     C_\cv(U\setminus S),
  \label{eq:toshow}
  \end{equation}
where $k=\dim L$. To this end we may assume by Proposition
\ref{prop:invariant} that $L'=L^\perp$, i.e., $L'$ is
the orthogonal complement of $L$. Then, as $U\subset L\cup L^\perp$ we
can write for any $b\in\R^m_{>0}$ the polytope $P(U,b)$ as the direct sum (see, e.g., \cite[][Lem.~3.1]{henk2014cone},
\cite[][Prop.~3.5]{BoeroeczkyHenk2016a} )
\begin{equation}
  P(U,b)=P(S,b_S)\oplus P(\ov S,b_{\ov S})
 \label{eq:directsumpoly}
     \end{equation}
     with $\ov S=U\setminus S$ and both polytopes are contained in
     orthogonal subspaces.
     For $Q\in\{U,S,\ov S\}$, and $u\in Q$  let
     $F_Q(u)=P(Q,b_Q)\cap\{x\in \lin Q: \langle u,x\rangle =
     b_{\{u\}}\}$ be the possible facet in direction $u$ of $P(Q,b_Q)$.
     Then for $u\in S$, $F_U(u)$ is a facet of $P(U,b)$ if and only if
     $F_S(u)$ is a facet of $P(S,b_s)$ and then $F_U(u)=F_S(u)\oplus P(\ov
     S,b_{\ov S})$. Thus for $u\in S$ we have
     \begin{equation*}
        \frac{b_{\{u\}}}{n}\vol_{n-1}(F_U(u)) = \frac{k}{n}\frac{b_{\{ u\}}}{k}
        \vol_{k-1}(F_S(u))\vol_{n-k}(P(\ov S,b_{\ov S})).
     \end{equation*}
       The same, of course, holds true if we replace $S$ by $\ov
       S$, and so we have
       \begin{equation*}
          \gamma(U,b)= \frac{k}{n} \vol_{n-k}(P(\ov S,b_{\ov
            S}))\gamma(S,b_s) \oplus \frac{n-k}{n} \vol_{k}(P(S,b_{
            S}))\gamma(\ov S,b_{\ov s}).
       \end{equation*}
       With \eqref{eq:directsumpoly} we can write
       \begin{equation*}
         \frac{1}{\vol(P(U,b)}\gamma(U,b) =\frac{k}{n} \frac{1}{\vol_k(P_S,b_S)}
         \gamma(S,b_S) \oplus \frac{n-k}{n} \frac{1}{\vol_{n-k}(P_{\ov
             S},b_{\ov S})} \gamma(\ov S,b_{\ov S}),
       \end{equation*}
and in view of \eqref{eq:scaling}, this shows \eqref{eq:toshow}.
\end{proof}

By Proposition \ref{prop:cvdirect}, Proposition
\ref{prop:sccdirect} and Theorem \ref{thm:inclusion} we have
\begin{equation*}
   \aff C_\cv(U)=\aff P_\scc(U)
\end{equation*}
and thus we get
\begin{corollary}
    \label{PropositionPconeAffineCombinationOfPscc}
    Let $U \in\mathcal{U}(n,m)$, $\gamma\in C_\cv(U)$ and $l=\dim P_{\scc}(U)$. Then
    there exists $\gamma_1,\dots,\gamma_{l+1}\in C_{\cv}(U)\cap \relint
    P_{\scc}(U)$, $\alpha_1,\dots,\alpha_{l+1}\in\R$, $\sum_{i=1}^{l+1}
    \alpha_i=1$, such that
    \begin{equation*}
       \gamma=\sum_{i=1}^{l+1}\alpha_i\,\gamma_i.
    \end{equation*}
\end{corollary}

\noindent
{\it Example} \ref{ex:general} (Polytopes in general positions, e.g.,
simplex){\it continued}. Theorem
\ref{thm:zhu} shows that for $U=(u_1,\dots,u_m)$ in general position
we have
\begin{equation}
  C_\cv(U)\cap\R^m_{>0}=\relint \conv\{e_1,\dots,e_m\}.
\label{eq:general}
\end{equation}
For a simplex, i.e., $m=n+1$, this is easy to see.
To this end we observe that by the existence theorem of Minkowski
(see, e.g., \cite[Theorem 8.2.1]{Schneider2014})
there exists an unique -- up to
translations -- simplex \begin{equation*}
  T(b^*)=\{ x\in\R^n :
  \langle u_i,  x\rangle\leq b_i^*, 1\leq i\leq n+1\}
\end{equation*} with $\vol(T)=1$.  Let $\phi_i$ be the
$(n-1)$-dimensional volume of the facet with outer unit normal vector
$u_i$. For
  $\gamma=(\gamma_1,\dots,\gamma_{n+1})^\intercal
  \in\conv\{e_1,\dots, e_{n+1}\}$, let $b_i=n\,\gamma_i/\phi_i$,
  $1\leq i\leq n+1$. Then  $T(b)$ is a simplex containing the
origin, with $\gamma(U,b)=\gamma$, and thus $\vol(T)=1$.   As we
always have $C_\cv(U)\subseteq \conv\{e_1,\dots,e_m\}$ this gives
\eqref{eq:general} for $m=n+1$. \hfill $\triangle$

\medskip\noindent
{\it Example} \ref{ex:parallel} (Parallelepiped){\it continued}.
Let $S_j=\{u_j,u_{n+j}\}=\{u_j,-u_j\}$, $j=1,\dots,n$, be the irreducible sets
of $U^s$. Then
\begin{equation*}
         C_\cv(S_j)=\conv\{e_j,e_{n+j}\},
\end{equation*}
and with Proposition \ref{prop:cvdirect} we obtain $C_\cv(U^s)=P_\scc(U^s)$.
Note that we consider the general cone-volume set without restricting
to the symmetric case (see Theorem \ref{thm:equality}). \hfill $\triangle$

\medskip\noindent
{\it Example} \ref{ex:trapez} (Trapezoid){\it continued}.
 From
\cite[Theorem 2.14]{Subspace_Concentration_of_Geometric_Measures} we
get that
\begin{equation}
   \begin{split}
   C_\cv(U)&\cap\R^4_{>0}   = \left\{ \gamma \in \R_{> 0}^4: \sum_{i=1}^4 \gamma_i = 1, \gamma_1 + \gamma_3 < \gamma_2 + \gamma_4 \right\}  \\
        & \cup \left\{ \gamma \in \R_{> 0}^4: \sum_{i=1}^4 \gamma_i =
          1, \gamma_1 + \gamma_3 \geq  \gamma_2 + \gamma_4 \geq 2
          \sqrt{\gamma_1 \gamma_3} \text{ and } \gamma_1 < \gamma_3
          \right\}.
        \end{split}
   \label{eq:trapeziod-hannes}
   \end{equation}
   \medskip
\hfill $\triangle$

   Trapezoids also serve as an example in
   order to show the following properties of $C_\cv(U)$.

   \begin{proposition} There exists $U\in\mathcal{U}(n,m)$ such that
   \begin{enumerate}
    \item $C_\cv(U)$ is not convex.
    \item $C_\cv(U)\cap\R^{|U|}_{>0} \not\subseteq \relint\left(C_\cv(U)\right)$,
    \item $C_\cv(U)$ is not closed.

\end{enumerate}
   \end{proposition}
   \begin{proof} We will only present $2$-dimensional examples, which
     can, however, easily be extended to any dimension via Proposition
    \ref{prop:cvdirect}.

     For i) and ii) we take the trapezoid from Example
     \ref{ex:trapez}, and let $A$ be the first set of the
     (disjoint) union
     \eqref{eq:trapeziod-hannes} and $B$ the second one. Obviously,
     $B$ is not convex and Figure \ref{FigurePconeTrapezoid} presents
     a visualization. 
\begin{figure}[bh]
  \centering
  \hfill
        \begin{subfigure}{0.45\textwidth}
        
            \includegraphics[scale = 0.4]{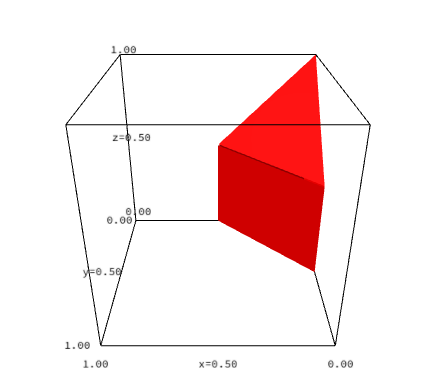}

        \caption{Subset $A$ of $C_\cv(U)\cap\R^4_{>0}$}
      \end{subfigure}
            \hfill
        \begin{subfigure}{0.45\textwidth}
            \includegraphics[scale = 0.4]{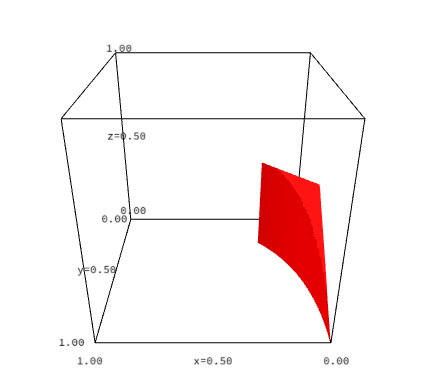}
            \caption{Subset $B$ of $C_\cv(U)\cap\R^4_{>0}$}
          \end{subfigure}
          \hfill
        \caption{The $x-$axis corresponds to $\gamma_1$, the $y-$axis to $\gamma_3$ and the $z-$axis to $\gamma_2$. The corresponding vector in $C_\cv(U)$ is given via the formula $(\gamma_1,\gamma_2,\gamma_3,1-(\gamma_1+\gamma_2+\gamma_3))$.}
        \label{FigurePconeTrapezoid}
    \end{figure}

    For ii) let $\gamma=(1/9,2/9,4/9,2/9)^\intercal$. Then
\begin{equation*}
  \gamma_1+\gamma_3>\gamma_2+\gamma_4=2\sqrt{\gamma_1\,\gamma_3},
  \gamma_1<\gamma_3, \text{ and } \gamma_1+\gamma_2+\gamma_3+\gamma_4=1,
\end{equation*}
and by \eqref{eq:trapeziod-hannes} we have $\gamma\in C_\cv(U)
\cap\R^m_{>0}$. However, for (small)
$\epsilon>0$ the vector
$\ov\gamma=(1/9+\epsilon,2/9-\epsilon,4/9,2/9)^\intercal$ still
satisfies all the above strict inequalities  but as
$\ov\gamma_2+\ov\gamma_4<2\sqrt{\ov\gamma_1\,\ov\gamma_3}$ it is
not contained in  $C_\cv(U)$. Hence, $\gamma\in\bd C_\cv(U)$.

To see iii), i.e., $C_\cv(U)$ is not closed in general we consider for
$U=(e_2,e_1+e_2,-e_2,-e_1+e_2)$ the trapezoids $P_U(b_\epsilon)$ with $b_\epsilon=(\epsilon,0,0,1/\epsilon+\epsilon)^\intercal$.
% \begin{equation*}
%             P_\epsilon=\{ 0\leq x_2\leq\epsilon, x_1+x_2\leq 0,
%             -x_1+x_2\leq \epsilon+1/\epsilon\}.
%           \end{equation*}

By elementary calculations we get for the cone-volume vector
$\gamma(\epsilon)=\gamma((U,b_\epsilon)$
\begin{equation*}
         \gamma(\epsilon)_1=\frac{1}{2}\left(1-\frac{\epsilon^2}{2}\right),
         \gamma(\epsilon)_2=\gamma(\epsilon)_3=0, \gamma(\epsilon)_4=\frac{1}{2}\left(1+\frac{\epsilon^2}{2}\right).
\end{equation*}
Hence, $(1/2,0,0,1/2)^\intercal\in\cl(C_\cv(U))$ but apparently there
is no right hand side $b$ such that the trapezoid $P_U(b)$ has this cone-volume vector.
\end{proof}

We remark that for $U\in\mathcal{U}(n,m)$ we always have
\begin{equation*}
  \relint C_\cv(U) \subset C_\cv(U)\cap\R^m_{>0}.
\end{equation*}

Next  we point out that $C_\cv(U)$ is path-connected.
\begin{proposition} Let $U\in\mathcal{U}(n,m)$. Then $C_\cv(U)$ is
  path-connected. 

\label{prop:pathconnected}
\end{proposition}
\begin{proof} Let $\gamma,\ov\gamma \in C_\cv(U)$ and let $b,\ov
  b\in \R^m$ such that $\gamma=\gamma(U, b)$, $\ov
  \gamma=\gamma(U,\ov b)$.  First we assume that both polytopes
  $P(U,b)$ and $P(U,\ov b)$ have (all) $m$ facets.

  Now let $F_i(b)=P(U,b)\cap\{x\in\R^n: \langle u_i,x\rangle =b_i\}$
  and $\phi_i=\vol_{n-1}(F_i(b))$ for $1\leq i\leq m$. 
 For $t\in\R^n$ we have
  $0\in (t+P(U,b))$  if and only if  $U^\intercal t \geq -b$ and for
  those $t$
  the cone-volume vector $\gamma(t)$ of  $t+P(U,b)$ is given by
  \begin{equation}
         \gamma(t)_i= \gamma_i+\frac{\phi_i}{n}\langle u_i,
         t\rangle,\quad 1\leq i\leq m.
 \label{eq:changecone}
       \end{equation}
  Let $t_0$ be chosen such that the centroid of $t_0+P(U,b)$ is at the
  origin and set $\alpha=\gamma(t_0)$.  From \eqref{eq:changecone} we
  have for $\lambda\in [0,1]$ that
  $(1-\lambda)\,\gamma+\lambda\,\alpha=\gamma(\lambda\,t_0)$ and so
  $\conv\{\gamma,\alpha\}\subset C_\cv(U)$.

  In the same way we choose $\beta$ with respect to $\ov\gamma$. The
  two polytopes $P(U,\alpha)$ and $P(U,\beta)$ have their
  centroids at the origin and by \cite[Thm.~I]{henk2014cone} we know
  $\alpha,\beta\in\relint P_\scc(U)$. As $ P_\scc(U)$ is convex
  and with Theorem \ref{thm:inclusion} we conclude
  $\conv\{\alpha,\beta\}\subset C_\cv(U)$. Hence we have found a path
  inside $C_\cv(U)$ connecting $\gamma$ and $\ov\gamma$.

 Now assume that for $u_1$, say, $F_1(b)$ is not a facet of
 $P(U,b)$, where we may assume that $b_1=\max\{\langle u_1, x\rangle: x\in
 P(U,b)\}$. Moreover, by the above argument, we may assume \(b_1 > 0\); otherwise, we translate \(P(U,b)\) so that \(0 \in \mathrm{int}(P(U,b))\). For $\epsilon>0$ let $b_\epsilon=b-\epsilon e_1$ and
 $F_i(b_\epsilon)=P(U,b_\epsilon)\cap\{x\in\R^n: \langle u_i,x\rangle
 =(b_\epsilon)_i\}$, $1\leq i\leq m$. Then for all sufficiently small
 $\epsilon >0$, $F_i(b_\epsilon)$ is a facet of $P(U,b_\epsilon)$ if
 $F_i(b)$ is a facet of $P(U,b)$, and in addition  $F_1(b_\epsilon)$
 is a facet of $P(U,b_\epsilon)$. Moreover, for small $\epsilon$ the
 $(n-1)$-dimensional volumes of the facets
 as well as  the cone volumes depends on $\epsilon$ in a polynomial way, and
 so does the volume of $P(U,b_\epsilon)$. Hence there exists a path
 in $C_\cv(U)$ connecting $\gamma(U,b)$ and $\gamma(U,b_\epsilon)$ for
 a small positive $\epsilon$. Iterating the process we can always find
 a path in $C_\cv(U)$ connecting  $\gamma(U,b)$ with the cone-volume
 vector of a polytope having all $m$ facets. Together with the first
 discussed case we are done. 
\end{proof}

In the example of a parallelepiped we have seen that for
$U^s\in\mathcal{U}(n,2n)$ we have $C_\cv(U)=P_\scc(U)$. This is also
essentially the only case as claimed in Theorem \ref{thm:equality}.
%\begin{theorem} Let $U\in\mathcal{U}(n,m)$. Then 
%  $C_\cv(U)=P_\scc(U)$ if and only if  $m=2n$ and up to renumbering we have
%  $u_{n+i}=-u_i$, $1\leq i\leq n$.
%\label{thm:equality}  
%\end{theorem}
\begin{proof}[Proof of Theorem \ref{thm:equality}] It remains to prove the necessity part. For a vector $v\in\R^m$ let $|v|_0=|\{i\in\{1,\dots,m\}
  :v_i\ne 0\}|$ be the cardinality of its non-zero coordinates. As 
  $P_\scc(U)$ is (up to scaling) the basis matroid polytope we have
  $|v|_0\geq n$ for any $v\in P_\scc(U)$.

  Let us firstly assume that there
  exists a subset  $S\subseteq U$ with $|S|\leq 2n-1$
  and $\pos S=\R^n$. Then there exists a  $\ov b_S\in\R^{|S|}_{\geq 0}$ such that
  $P(S,\ov b_s)$ is a $n$-dimensional polytope with facets in the
  directions $u\in S$. Now let $b\in\R^m$ such that
   \begin{equation*}
          P(U,b)=P(S,\ov b_S), 
   \end{equation*}
e.g., we set for $u\notin S$, $b_{\{u\}}=\max\{\langle u, x\rangle,
x\in P(S,\bar b_S)\}$. Then
$P(U,b)$ is an $n$-polytope with $2n-1$ facets. Moving $P(U,b)$ such
that a vertex is the origin, yields a polytope $P(U,\widetilde{b})$ with
\begin{equation*}
       |\gamma(U,\widetilde{b})|_0 \leq |S|-n <n.
     \end{equation*}
Hence this shows that if $C_\cv(U)=P_\scc(U)$ then any positive basis
$S$ of $\R^n$ contained in $U$ must have cardinality $\geq 2n$. Here
a set $S$ of vectors
build a positive basis of $\R^n$ if $\pos S=\R^n$ and 
for any strict subset $\bar S\subsetneq S$ we have $\pos \bar S\subsetneq\R^n$.    From the theory of positive bases it is known that we have $n+1\leq |S|\leq
2n$ (see, e.g.,  \cite[Theorem 6.6]{Regis2015a}).

As $\pos U=\R^n$, $U$ contains positive bases and
we have shown that all of them have cardinality $2n$. Let $S\subseteq
U$ be such a basis of cardinality $2n$. By the characterisation of
those maximal positive bases \cite[Theorem 6.3]{Regis2015a} we conclude in our setting that up to
renumbering $S=(V,-V)$, where $V=(v_1,\dots,v_n)$ is a set of $n$
linearly independent unit vectors.  

Suppose there exists an $u\in U\setminus S$ and without loss of
generality let
\begin{equation*}
            u=\sum_{i=1}^l \rho_i v_i 
          \end{equation*}
          with $\rho_i>0$ and $2\leq l\leq n$. Then it is not hard to see that
          \begin{equation*}
            -v_1,\dots,-v_l, u ,\pm v_{l+1},\dots,\pm  v_n
          \end{equation*}
          also build a positive basis, but of cardinality less than $2n$.
Hence, it follows $U=S$.        
\end{proof}

For later purpose we also point out that the cone-volume vectors are
continuous.

\begin{lemma} Let $b^{(j)},b\in\R^m_{\geq 0}$, $j\in\N$,  with
  $\lim_{j\to\infty} b^{(j)} = b$. Then
  \begin{equation*}
    \lim_{j\to\infty} \gamma(U,b^{(j)})= \gamma(U,b).
  \end{equation*}
\label{lem:continuous}
\end{lemma}
\begin{proof} As $b^{(j)}\to b$  we have $P(U,b^{(j)})\to P(U,b)$ in
  the Hausdorff metric. As $\vol(P(U,b^{(j)}))\to \vol(P(U,b))$ it
  suffices to consider the case $\vol(P(U,b))>0$. Hence, $P(U,b)$ has facets and let $S\subseteq
  U$ be the vectors corresponding to these facets. Moreover, for
  $c\in\R^m$ and $u\in U$ let
  \begin{equation*}
     F(u,c)= P(U,c)\cap\{x\in\R^n : \langle u,x\rangle =c_{\{u\}}\}.
  \end{equation*}
  Then for $u\in S$  and 
 for all sufficiently large $j$, $F(u,b^{(j)})$ is a facet of
 $P(U,b^{(j)})$ and  $F(u,b^{(j)})\to F(u,b)$. Thus 
\begin{equation*}
            \lim_{j\to\infty}  \gamma(U, b^{(j)})_{\{u\}}=  \lim_{j\to\infty} \frac{1}{n} b^{(j)}_{\{u\}}\,\vol_{n-1}(F(u,b^{(j)}))=\gamma(U,b)_{\{u\}}.
\end{equation*}
As $\vol(P(U, b^{(j)})) \to\vol(P(U,b))=\sum_{u\in S} \gamma(U,b)_{\{u\}}$ we conclude for $u\in U\setminus S$
\begin{equation*}
            \lim_{j\to\infty}  \gamma(U, b^{(j)})_{\{u\}}= 0=\gamma(U,b)_{\{u\}}.
\end{equation*}
\end{proof}

Finally, we remark that it is also possible to consider instead of the
polytope $P_\scc(U)$ the half open subspace concentration set 
     \begin{equation*}
    \begin{split} 
    \widehat{P_\scc} (U)=  \Biggl\{ x \in \R^m & : \sum_{i = 1}^m x_i = 1, \sum_{u_i \in S }  x_i = \frac{\rg(S)}{n},  S \in \mathcal{F}(U), \\
     &\, x\geq 0, \sum_{u_i \in  S }  x_i < \frac{\rg(S)}{n},\,  S \in
     \mathcal{L}(U) \setminus \mathcal{F}(U) \Biggl\}.
   \end{split}
 \end{equation*}
Then it can be shown that relations in Theorem \ref{thm:blyz} and Theorem
\ref{thm:inclusion} become 
  \begin{align*}
     C^s_\cv(U^s) & =\widehat{P_\scc} (U^s) \cap 
                    \{ x \in \R^m : x_{i } = x_{i + m'},\,1\leq i\leq m' \}, % =: \widehat{P_\scc}^s (U^s) \quad \text{ as well as } 
    \\ \quad C_\cv(U)&\supseteq \widehat{P_\scc}(U).
   \end{align*}
However, we prefer to work with the closed polytope $P_\scc(U)$.
                                                                                                       
    \section{Semialgebraic Sets and Cone-Volumes}
\label{Section:SemialgebraicSets}
First we recall that a semialgebraic set in $\R^m$ is the finite union of sets of the form
\begin{equation*}
                      \{ x\in\R^m : f_i(x)\geq 0, i\in I, g_j(x)>0, j\in J \}
\end{equation*}
where $I,J\subseteq \N$ are finite and $f_i,g_j\in\R[x]$ are
polynomials. By definition, the finite union of semialgebraic sets is
semialgebraic and by the classical Tarski-Seidenberg principle the
projection of a semialgebraic set is again a semialgebraic set, see
\cite[Theorem 2.2.1.]{bochnak2013real}.

In order to show that $C_\cv(U)$ is a semialgebraic set, we will first focus on the cone-volume vectors of $n$-polytopes $P(U,b)$ which are simple and strongly isomorphic. An $n$-dimensional polytope is called  simple if each vertex is contained in exactly $n$ facets, and two polytopes are strongly isomorphic if their face lattice is isomorphic and the affine hulls of the facets are parallel (cf.~\cite{Schneider2014}). In our setting this implies that if for $b,\ov b\in\R^m_{\geq 0}$ two $n$-polytopes  $P(U,b)$ and $P(U,\ov b)$ are combinatorially isomorphic, i.e., their face lattices are isomorphic, then they are also strongly isomorphic.

Observe, for  a''generic'' right hand side vector $b\in\R^m_{\geq 0}$ the polytope $P(U,b)$ is simple and the general case will be deduced from the simple case by approximation. The next lemma collects some well-known properties of strongly isomorphic simple polytopes.
\begin{lemma} Let $U\in\mathcal{U}(n,m)$.  Then  $\R^m_{\geq 0}$ can be subdivided into finally many polyhedral $m$-dimensional cones $A_1(U), \dots, A_l(U)$ such that
    \begin{enumerate}
    \item    For $b,\ov b\in\mathrm{int}(A_k(U))$ the polytopes $P(U,b)$, $P(U,\ov b)$ are strongly isomorphic, simple and $n$-dimensional.
    \item  There exist polynomials $v_k(y), f_{k,i}(y)\in\R[y_1,\dots,y_m]$, $1\leq k\leq l$, and $1\leq i\leq m$ of degree at most $n$, such that for $b\in  \mathrm{int}(A_k(U))$
      \begin{equation*}
        \begin{split}
          v_k(b &)=\vol(P(U,b)), \\
          f_{k,i}(b&)=\vol_{n-1}\left(P(U,b)\cap\{x\in \R^n: \langle u_i, x\rangle=b_i\}\right).
        \end{split}
      \end{equation*}
    \end{enumerate}
\label{lem:mainsemi}
\end{lemma}
\begin{proof} The existence of cones \( A_1(U), \dots, A_l(U) \) satisfying i) follows from McMullen's representation theorem for convex polytopes \cite{mcmullen1973representations}. These cones partition the parameter space into regions corresponding to distinct combinatorial types.

  To address ii), we consider one fixed cone \( A_k(U) \). The interior points  correspond to a fixed specific combinatorial type, and let $S\subseteq U$ be the subset of vectors corresponding to the facets of this combinatorial type.
  Then for $b\in \inter A_k(U)$  we have $P(U,b)=P(S,b_S)$  and we may write
  \begin{equation}
    \vol(P(U,b))=% \frac{1}{n}\sum_{u\in U} b_{\{u\}}\, \vol_{n-1}(F(u,b))
     \frac{1}{n}\sum_{u\in S} b_{\{u\}}\, \vol_{n-1}(F(u,b)),
\label{eq:sumpoly}
\end{equation}
where
\begin{equation*}
  F(u,b)=P(U,b)\cap\{x\in \R^n: \langle u_i, x\rangle=b_{\{u\}}\}.
\end{equation*}
For $u\in S$, $b_{\{u\}}$ is the so called support number of $P(U,b)=P(S,b_S)$ in direction $u$, i.e.,
\begin{equation*}
   b_{\{u\}}=\sup\{\langle u, x\rangle: x\in P(U,b)\}
\end{equation*}
and for $u\in U\setminus S$ we have $ \vol_{n-1}(F(u,b)) =0$. The proof of \cite[Lemma 5.1.3]{Schneider2014}, more precisely the equations (5.6) and (5.7), now show that for $u\in S$ the volume $\vol_{n-1}(F(u,b))$ is a polynomial  of degree $n-1$ in the coordinates of $b_S$. For $u\in U\setminus S$,  $\vol_{n-1}(F(u,b))$ is just the null polynomial and with \eqref{eq:sumpoly} the assertion follows.
\end{proof}

The cones $A_1(U),\dots,A_l(U)$ are called the type-cones of $U$, cf. \cite{mcmullen1973representations}, and next we investigate them for our running examples.

\medskip\noindent
{\it Example} \ref{ex:general} (Simplex){\it continued}.
Let $U \in\mathcal{U}(n,n+1)$ be in general position. Then every $b\in\R^m_{\geq 0}$, $b\ne 0$,  $P(U,b)$ is an $n$-dimensional simplex. Thus there is only one type cone $A_1(U)=\R^m_{\geq 0}$, and the volume of $P(U,b)$ or of its facets can easily be calculated via determinants, which then are polynomials in the coordinates of $b$.   \hfill$\triangle$

\medskip\noindent
{\it Example} \ref{ex:parallel} (Parallelepiped) {\it continued}. Here again we have only one type-cone $A_1(U^s) = \R_{\geq 0}^m$. This follows from the fact that two $n$-polytopes $P(U^s,b)$ and $P(U^s,\tilde{b})$, where $b, \tilde{b} \in\R^m_{\geq 0}$ are $n$-dimensional parallelepipeds with the same facet directions. The volume polynomial is given by
\begin{equation*}
   \vol(P(U^s,b^s))= \frac{ \prod_{i=1}^n (b_i+b_{n+i})}{ |\det(A)|},
\end{equation*}
where $A\subset U^s$ consists of the first $n$ columns. \hfill$\triangle$

% \vspace{\baselineskip}

% These two examples show that the cube and the simplex are monotypic, cf. Definition \ref{Definition:monotypicOuterNormals}.

% \vspace{\baselineskip}
\medskip
\noindent
{\it Example}  \ref{ex:trapez} (Trapezoid) {\it continued}.
% Next, we want to consider the outer normal vectors $U$ of a trapezoid and want to compute the type-cones. This subsequent example illustrates that, in general, there can be more than just one type-cone. Moreover, a polytope $P(U,b)$ with right-hand side $b \in \bd A_k(U)$, for one type-cone $A_k(U)$, can still be simple. In the plane that is clear, since every polygon is simple.
We assume that the vectors in  $U$ are ordered counter-clockwise, and
$u_1 = e_2= - u_3$, see Figure \ref{fig:typeConesTrapezoid}. 
\begin{figure}[hbt]
    \centering
\begin{tikzpicture}[scale=0.8]
     \begin{scope}

        \filldraw[color=blue!60, fill = blue!5, very thick](0,0) circle (1);
        \draw[->,very thick](0,0) -- (0,1) node[above]{$u_1$};
        \draw[->,very thick](0,0) -- (-0.94,0.31) node[left]{$u_2$};
        \draw[->,very thick](0,0) -- (0,-1) node[below]{$u_3$};
        \draw[->,very thick](0,0) -- (0.98,0.16) node[right]{$u_4$};
        \draw [red,thick,domain=180:360] plot ({cos(\x)}, {sin(\x)});
        \end{scope}
\hfill
       \begin{scope}[xshift=4.5cm,scale = 0.85]
       \draw[black] (1.18,-1.0)  -- (0.85,1.0)  -- (-0.73,1.0)  -- (-1.39,-1.0)  -- cycle;
       \draw[dotted,thick] (-0.73,1.0) -- (0.59,5.0);
       \draw[dotted,thick] (0.85,1.0) -- (0.19,5.0);
       \draw[dotted,thick] (-2,1.0) -- (-1,1) node[above left]{$H^+(u_1,b_1)$} -- (2,1.0);
       \filldraw[color = red!100, thick](0.31, 4.2) circle (0.08);
       \draw[->,very thick] (0.06,1) -- (0.06,2) node[right]{$u_1$};
       \draw[->,very thick] (1,0) -- (1.98,0.16) node[right]{$u_4$};
       \draw[->,very thick] (-1.05,0) -- (-2.04,0.31) node[left]{$u_2$};

       \end{scope} \hfill
       \begin{scope}[xshift=9cm,scale = 0.5]

       \draw[black] (1.18,-2.0)   -- (-1.39,-2.0)  -- (0.33,4.22)  -- cycle;

       \draw[dotted,thick] (-3,7) node[above]{$H^+(u_1,b_1)$} -- (3,7)  ;
       \draw[->,very thick] (0.5,7) -- (0.5,9) node[right]{$u_1$};

       \draw[dotted,thick] (1.18,-2.0) -- (-0.52,10.44);
       \draw[dotted,thick] (-1.39,-2.0) -- (2.05,10.44);
       \filldraw[color = red!100, thick](0.33,4.22) circle (0.14);
       \draw[->,very thick] (0.75,1.11) -- (2.71,1.43) node[right]{$u_4$};
       \draw[->,very thick] (-0.53,1.11) -- (-2.41,1.73) node[left]{$u_2$};
       \end{scope}
    \end{tikzpicture}
    \caption{Illustration of the two scenarios for type-cones in the trapezoid case, with the outer unit normal vector set $U$ drawn on the left. Center: The intersection of the two non-parallel lines occurs above the line defined by $u_1$, forming a trapezoid. Right: The intersection occurs below the line defined by $u_1$, producing a triangle.}
    \label{fig:typeConesTrapezoid}
\end{figure}
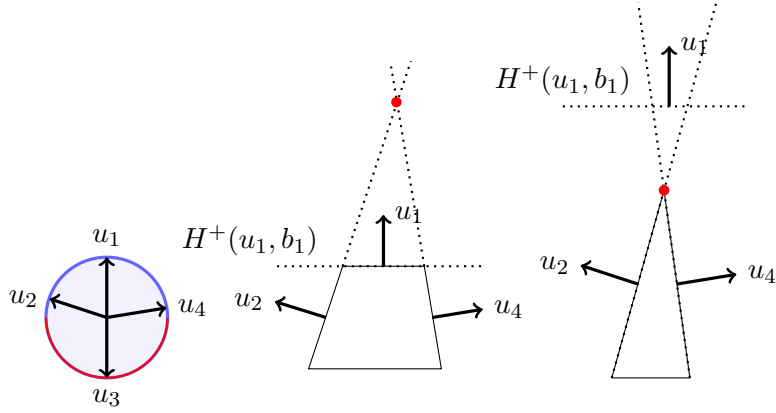
Let $a_2, a_4 \in \R$, $a_4 > 0 > a_2$ such that $u_2 =
\frac{1}{l_2}(-1,-a_2)^T$ and $u_4 = \frac{1}{l_4}(1,a_4)^T$, with $l_2 = \| (-1,-a_2) \|$ and $l_4 = \| (1,a_4)
\|$. For  $b \in \R_{>0}$, the polygon $P = P(U,b)$ is a
$2$-dimensional trapezoid if and only if the intersection point of the lines $H(u_2,b_2)$ and $H(u_4,b_4)$ is above the line  $H(u_1,b_1)$, where $H(u_i,b_i)=\{x\in\R^2 : \langle u_i,x\rangle=b_i\}$. This intersection point is given by
\begin{equation}
    \frac{1}{a_4-a_2} \begin{bmatrix}
        -(l_2 a_4 b_2 + l_4 a_2 b_4) \\
        l_2 b_2 + l_4 b_4
    \end{bmatrix}
\end{equation}
and so we have a $2$-dimensional trapezoid if and only if $l_2 b_2 + l_4 b_4 > (a_4 - a_2) b_1$. Thus, the hyperplane separating the two different type-cones is given by \[H= \{ b \in \R_{\geq 0}^4 : l_2 b_2 + l_4 b_4 - (a_4 - a_2) b_1 = 0 \} \]
and the two different type-cones are
\begin{align*}
    A_1(U) & = \{ b \in \R_{\geq 0}^4 : l_2 b_2 + l_4 b_4 - (a_4 - a_2) b_1 \geq 0 \} \text{ and } \\ A_2(U) & = \{ b \in \R_{\geq 0}^4 : l_2 b_2 + l_4 b_4 - (a_4 - a_2) b_1 \leq 0 \}.
\end{align*}
For $b\in \inter A_1(U)$ we get  $2$-dimensional trapezoids, and $2$-dimensional triangles for $b\in  \inter A_2(U)$ as well as for $b\in\R^4_{>0}\cap H$. \hfill$\triangle$

\vspace{\baselineskip}
In the following, we assume that for $U\in\mathcal{U}(n,m)$ and
 $k\in\{1,\cdots,l\}$ the polyhedral cone  $A_k(U)$ from Lemma \ref{lem:mainsemi} is given by
\begin{equation*}
               A_k(U)=\{b\in\R^m: B_k \cdot b\geq 0\}
\end{equation*}
for some matrix $B_k\in\R^{m_k\times m}$. For the  computation of such a matrix $B_k$, we refer to  \cite{fillastre2017shapes}. In view of Lemma \ref{lem:mainsemi} we set for $1\leq k\leq l$
\begin{equation}
  W_k(U)=\left\{ (\gamma,b) \in\R^m \times \R^m : B_k\,b >0,  v_k(b)=1,
    \gamma_i=  \frac{f_{k,i}(b)\cdot b_i}{n}, 1\leq i\leq m\right\}.
\label{eq:conevolumevectorstype}  
\end{equation}
Observe that for $(\gamma,b)\in  W_k(U)$, we have $b\in\inter A_k(U)$ and so we know by Lemma \ref{lem:mainsemi}  that $P(U,b)$ is a simple polytope with cone-volume vector
\begin{equation*}
                  \gamma(U,b)=\gamma.
\end{equation*}
Apparently, $W_k(U)$ is a semi-algebraic set and they are the main
ingredients of the proof of  Theorem \ref{thm:main} which will
immediately follow from the next lemma.

\begin{lemma} Let $U\in\mathcal{U}(n,m)$. Then
  \begin{equation}
  \left\{   \big(\gamma(U,b), b\big) : b\in\R^m_{\geq 0} \text{ with }\gamma(U,b)\in C_\cv(U)     \right\} =\bigcup_{k=1}^l \cl W_k(U).
\label{eq:mainset}
\end{equation}% \begin{equation}
%   \left\{   \big(\gamma(U,b), b\big) : b\in\R^m_{\geq 0} \text{ with }\gamma(U,b)\in C_\cv(U)     \right\} =\bigcup_{k=1}^l \cl W_k(U).
% \label{eq:mainset}
%\end{equation}
\label{lem:lemmasemi}
\end{lemma} 
\begin{proof} First let $b\in\R^m_{\geq 0}$ with  $\gamma(U,b)\in
  C_\cv(U)$, and we show that there exists a $k\in\{1,\dots,l\}$ such that
\begin{equation*}
         \big(\gamma(U,b), b\big)\in \cl W_k(U).
       \end{equation*}
By definition $P(U,b)$ is an $n$-dimensional polytope of volume 1. For $\epsilon\geq 0$ let $b(\epsilon)=b+(\epsilon,\epsilon^2,\cdots,\epsilon^m)$. Then, except for finitely many values of  $\epsilon\in[0,\infty)$ the vectors $b(\epsilon)$ are contained in the interior of the cones $A_i(U)$, $i\in\{1,\dots,l\}$. So there exists a sequence  $\epsilon_j\in (0,\infty)$, $j\in\N$, with $\lim_{j\to\infty} \epsilon_j= 0$ and $b(\epsilon_j)\in \inter( A_k(U))$, say.
  Then $P(U,b(\epsilon_j))\to P(U,b)$ in the Hausdorff metric and so $\vol(P(U,b(\epsilon_j))\to 1$. Hence, with $\widetilde{b}(\epsilon_j)=(\vol(P(U,b(\epsilon_j)))^{-1/n}b(\epsilon_j))\in\inter A_k(U)$  we have
  \begin{equation*}
       \big( \gamma(U, \widetilde{b}(\epsilon_j)), \widetilde{b}(\epsilon_j)\big)\in W_k(U).
     \end{equation*}
From Lemma \ref{lem:continuous} we get
     \begin{equation*}
   \lim_{j\to\infty} \gamma(U, \widetilde{b}(\epsilon_j)) = \gamma(U,b),
 \end{equation*}
 and so $(\gamma(U,b),b))\in\cl W_k(U)$.

 Next for the reverse inclusion let $(\gamma^{(j)},b^{(j)})\in W_k(U)$, $j\in\N$, with $(\gamma^{(j)},b^{(j)})\to  (\gamma,b)$.
  Then $\gamma^{(j)}=\gamma(U,b^{(j)})$ and again by Lemma \ref{lem:continuous} we conclude
  \begin{equation*}
    \gamma(U,b)=\lim_{j\to\infty} \gamma(U,b^{(j)}) =\lim_{j\to\infty}\gamma^{(j)} = \gamma.
  \end{equation*}
\end{proof}

Apparently, Lemma \ref{lem:lemmasemi} % and Lemma \ref{lem:lemma2} show \eqref{eq:mainset} which immediately
implies Theorem \ref{thm:main}.

\begin{proof}[Proof of Theorem \ref{thm:main}] Let $\Pi:\R^m\times\R^m\to\R^m$ be the projection $\Pi(x,y)=x$. On account of \ref{lem:lemmasemi}  we have
  \begin{equation}
    C_\cv(U)=\Pi\left(\bigcup_{k=1}^l \cl W_k(U)\right) = \bigcup_{k=1}^l\Pi( \cl W_k(U)).
  \label{eq:coneprojection}
  \end{equation}
As $W_k(U)$ is a semialgebraic set, the Tarski-Seidenberg principle implies that $\cl W_k(U)$ and then $\Pi(\cl  W_k(U))$ are  semialgebraic as well, and so is $C_\cv(U)$.
\end{proof}

\begin{remark} By the Tarski-Seidenberg principle \cite[Theorem 2.2.1.]{bochnak2013real} we also get that $C_\cv(U)\cap\R^m_{>0}$ and $\relint C_\cv(U)$ are semialgebraic.
\end{remark}

\begin{corollary} Let $U\in\mathcal{U}(n,m)$ with $l$ type-cones. The semialgebraic set
  $C_\cv(U)$ can be described by at most $l (2(m+1))^3 n^{m \cdot \mathcal{O}(1)}$ polynomials in $m$ variables of
  degree at most $n^{\mathcal{O}(m^2)}$.  
\label{cor:degree}  
\end{corollary}%\todo{Do we have a bound on the number of polynomials?}
\begin{proof} Notice that we can write
\begin{align}
    \cl(W_k(U))  = \\ \Bigl\{  (x,y) \in \R^m \times \R^m \, : \, \forall t \, \in \R \, \exists (p,q) \in \R^m & \times \R^m 
    :  \big((p,q) \in W_k(U)  \text{ and } \\  & \| (x,y) - (p,q) \|^2  < t^2 \big) \text{ or } [t = 0] \Bigl\}.
\end{align}
Now the bound for the number of polynomials and the degree is a conclusion of this representation together with \cite[Theorem 1]{basu1997improved} and the fact that all polynomials appearing in $W_k(U)$ have degree at most $n$.
\end{proof}

\medskip\noindent
{\it Example} \ref{ex:general} (Simplex) {\it continued}. By the existence theorem of Minkowski all simplices $P(U,b)$
  of volume $1$ are translates of each other. With $\alpha_i=\vol(P(U,e_i))$, $1\leq i\leq n+1$, we conclude
  \begin{equation*}
          \cl W_1(U) = \conv\left\{(e_i, \alpha_i^{-1/n}e_i): 1\leq i\leq n+1\right\}.
  \end{equation*}
\hfill $\triangle$

So in this case $\cl W_1(U)$ is convex, which is, of course, not true in general, as already a parallelepiped shows.

\medskip\noindent
{\it Example} \ref{ex:parallel} (Parallelepiped) {continued}. We
have seen already that the first $n$ coordinates of $\cl(W_1(U^s))$ are convex,
in the sense that its projection onto the first $n$ coordinates maps the set $\cl(W_1(U^s))$ onto $C_{\cv}(U^s)$ which is equal to $P_\scc(U^s)$. However, the last $n$ coordinates do not behave convex: take two right hand sides $b,\tilde{b}\in\R^{2n}_{\geq 0}$ such that the $n$-parallelepipeds $P(U,b)$, $P(U,\tilde{b})$ have volume $1$ but are not homothetic. Then by the Brunn-Minkowski theorem $P(U, \frac{1}{2}b+\frac{1}{2}\tilde{b})$ has volume greater than 1.
\hfill $\triangle$

Generalizing the observation made by the parallelepiped, we obtain the following characterization.

\begin{proposition} Let $U\in\mathcal{U}(n,m)$. Then $\cup_{k=1}^l \cl W_k(U)$ is convex if and only if $m=n+1$.
\end{proposition}
\begin{proof} Assume that $\cup_{k=1}^l \cl W_k(U)$ is convex. Then the Brunn-Minkowski theorem, as used in the example of the parallelepiped, shows that all $n$-polytopes $P(u,b)$, $b\in\R^m_{\geq 0}$, of volume $1$ are translates of each other, and, in particular,  the vectors $(\vol_{n-1}( F(u,b)) : u\in U) \in\R^m$ are the same for all those $b$s. By the existence theorem of Minkowski this implies that $\dim\, \mathrm{kern} U =1$ and so $m=n+1$.   For $m=n+1$  see the example of a simplex.
\end{proof}

In order to improve the representation of the $C_\cv(U)$  as a projection of the sets $\cup_{k=1}^l\cl(W_k(U))$
let us assume that $\mathcal{I}(U)\subseteq\{1,\dots,l\}$ be all indices such for $j\in   \mathcal{I}(U)$
and $b\in\inter A_j(U)$,  $F(u_i,b)$ is a facet of $P(U,b)$ for all $1\leq i\leq m$. In words, $\mathcal{I}(U)$ represents all the type cones $A_k(U)$ where for an interior vector $b$ all vectors $u_i$, $1\leq i\leq m$, are outer unit normal vectors of  facets of $P(U,b)$.

\begin{proposition} Let $U\in\mathcal{U}(n,m)$ and $\Pi :\R^{m}\times \R^m\to\R^m$ be the projection $\Pi(x,y)=x$. Then
  \begin{equation*}
      C_\cv(U)= \bigcup_{k\in \mathcal{I}(U)} \Pi(\cl W_k(U)).
    \end{equation*}
    \label{Proposition:PconeEqualsProjectionOfTypes}
\end{proposition}
\begin{proof} In view of Lemma \ref{lem:lemmasemi} we just have to show 
that for  $b\in\R^m_{\geq 0}$ with  $\gamma(U,b)\in
  C_\cv(U)$, and we show that there exists a $k\in \mathcal{I}(U)$ such that
\begin{equation*}
         \big(\gamma(U,b), b\big)\in \cl W_k(U).
       \end{equation*}
To this end let $\gamma(U,b)\in C_\cv(U)$ where we may assume that % $b_1,\dots,b_m$ are the support numbers of the $n$-polytope $P=P(U,b)$, i.e.,
  $b_i=\max\{\langle u_i, x\rangle, \linebreak  x\in P\}$.
  Increasing all the $b_i$s which correspond to facets of $P(U,b)$  by any small $\epsilon>0$ gives a polytope $P(U, \overline{b}(\epsilon))$
  where all vectors $u_i$ are now facet vectors and we may also assume $\vol(P(U, \overline{b}(\epsilon))=1$. Now we disturb $\overline{b}(\epsilon)$ a bit as in the proof of Lemma \ref{lem:lemmasemi} and derive at the same conclusion, but now with $k\in \mathcal{I}(U)$.
\end{proof}

\section{Polynomial Inequalities for Polygons}

In this section, we present for $U=(u_1,\dots,u_m)\in\mathcal{U}(2,m)$ a set of polynomials describing the set \(
W_k(U) \). In view of Corollary
\ref{Proposition:PconeEqualsProjectionOfTypes} we will do it only for
$k\in\mathcal{I}(U)$, i.e., we consider only type-cones $A_k(U)$ such
that $F(u_i,b)$, $1\leq i\leq m$, are facets (edges) for all
$b\in\inter A_k(U)$. Since $n=2$ there exists only one such type-cone,
which we will denote by $A(U)$ and the associated set with the cone-volume
vectors will be denoted by $W(U)$ (cf.~\eqref{eq:conevolumevectorstype}).

In order to describe $A(U)$ and $W(U)$ we assume that the 
unit vectors $u_1,\dots,u_m$ are ordered counter-clockwise.
Now  $u_1,\dots,u_m$ are the outer unit normal vectors of  edges of
$P(U,b)$ if and only if the intersection point $v_l$ of the two
(neighbouring) lines
$\{x\in\R^2 : \langle u_{l},x\rangle=b_{l}\}$ and  $\{x\in\R^2 : \langle u_{l+1},x\rangle=b_{l+1}\}$
is a vertex of $P(U,b)$. Here the indices are always calculated
$\bmod\, m$. Moreover, $v_l$ is a vertex of $P(U,b)$ if and only if
\begin{equation}
                 \langle u_i, v_l\rangle < b_i \text{ for } i\in\{1,\dots,m\}\setminus\{l,l+1\}.
\end{equation}
As the coordinates of $v_l$ depends linearly on $b_l, b_{l+1}$ the
inequalities above for $l=1,\dots,m$ describe the interior of
$A(u)$. Hence we have found a representation  $A(U)=\{b\in\R^m  :
B\,b\geq 0\}$ for some $B\in \R^{m(m-2)\times m}$.  

For $b\in\inter A(U)$ the volume of the facets $f_i(b)$, i.e., the length of the edges were already calculated 
by Stancu \cite[Remark 2.1]{stancu2002discrete} and % with the notation
% of Lemma \ref{lem:mainsemi}
we have
\begin{align}
    f_{i}(b) =  \vol_1(F_i(b))  = &- b_i \left( \frac{\langle u_i, u_{i+1} \rangle}{\sqrt{1 - (\langle u_i, u_{i+1} \rangle)^2}} + \frac{\langle u_{i-1}, u_{i} \rangle}{\sqrt{1 - (\langle u_{i-1}, u_{i} \rangle)^2}} \right)  \\
    & + \frac{b_{i+1}}{\sqrt{1 - (\langle u_i , u_{i+1} \rangle)^2}} + \frac{b_{i-1}}{\sqrt{1 - (\langle u_{i-1} , u_{i} \rangle)^2}}.
\end{align}
Along with pyramid formula
\begin{equation}
    \vol(P(U,b))=\sum_{i = 1}^m f_{i}(b) \cdot  \frac{b_i}{2}.
\end{equation}
we have obtained a representation of the set $W(U)$.

Although it is easy to get this description $W(U)$, the computation of
$C_\cv(U)=\Pi(\cl(W(U)))$ remains challenging, as we must eliminate as
many quantifiers as there are columns in $U$. For instance, if we
consider a $U$, the software \textit{Mathematica} \cite{Mathematica}
was unable to generate a quantifier-free output for the left set.
 Even for quadrilaterals the explicit descriptions of $C_\cv(U)$
obtained by  Liu, Lu, Sun and Xiong \cite{LiuLuSunEtAl2024} % and by
% Pollehn \cite[Section 2.4]{Pollehn2019} for trapeziods 
are quite involved.  

\section{On the non-uniquness of cone-volume vectors}

In this section, we briefly study for 
$U \in \mathcal{U}(n,m)$ and  $\gamma \in C_{\cv}(U) \cap \R_{>0}^m$  the set 
\begin{equation*}
    S(U, \gamma) = \{ b \in \R_{\geq 0}^m : \gamma(U, b) = \gamma \}, 
\end{equation*}
consisting of all right hand sides  $b \in \R_{\geq 0}^m$ yielding the same cone-volume vector $b$. In the case of the simplex, i.e., $m=n+1$, the cardinality of this set is clearly $1$ whereas in our example of the parallelepiped it is infinity. Hence, to study its size we use the dimension $\dim^*(S(U,\gamma))$, where  $\dim^*(S(U,\gamma))$ denotes the dimension of $S(U,\gamma)$ as a semialgebraic set, \cite[Definition 2.8.1]{bochnak2013real}. In particular,
the cardinality of $S(U,\gamma)$ is finite if and only if $\dim^*(S(U,\gamma))=0$. The next proposition gives a lower bound on  $\dim^*(S(U,\gamma))$. 

\begin{proposition}
  Let \( U \in \mathcal{U}(n,m) \), 
let $U=S_1\cup S_2 \cup
\cdots \cup S_d$ be the unique partition into irreducible sets and let $\gamma \in C_{\cv}(U) \cap \R_{>0}^m$. 
Then it holds
\begin{equation*}
    \dim^*(S(U, \gamma)) \geq d - 1 = m-\dim(C_\cv(U))-1.
\end{equation*}
\end{proposition}

\begin{proof} The last identity follows from Proposition \ref{prop:cvdirect}. For the proof of the inequality we use induction over the number $d$ of irreducible separators of $U$. If $d = 1$ there is nothing to prove. Therefore, assume $d \geq 1$ and let $V = U \setminus S_d$. Then (cf.~ \eqref{eq:toshow})
    \begin{equation*}
        C_\cv(U) = \frac{\dim(S_d)}{n} C_\cv(S_d) \oplus \frac{\dim(V)}{n} C_\cv(V).
    \end{equation*}
 So we may write  $\gamma = (\frac{\dim(S_d)}{n}\gamma_{S_d},\frac{\dim(V)}{n}\gamma_V)$ and  consider now an element $b = (b_{S_d},b_V) \in S(U,\gamma)$. %In particular it holds $b_V \in S(V,\frac{\dim(V)}{n}\gamma_V)$. 
    For any $\lambda > 0$ it holds 
    \begin{equation*}
        b_{\lambda} = \left(\lambda b_{S_d},\lambda^{-\frac{\dim(S_d)}{n - \dim(S_d)}} b_V \right) \in S(U,\gamma).
    \end{equation*}
    Thus, we conclude 
    \begin{equation*}
        \dim^* (S(U,\gamma)) \geq \dim^*(S(V,\gamma_V)) + 1 \geq (d-1) - 1 + 1 = d - 1.
    \end{equation*}
\end{proof}

We conjecture that equality holds in the proposition above.
\begin{conjecture}
  \label{Conjecture:NumberOfSol}
  Let \( U \in \mathcal{U}(n,m) \), 
let $U=S_1\cup S_2 \cup
\cdots \cup S_d$ be the unique partition into irreducible sets and let $\gamma \in C_{\cv}(U) \cap \R_{>0}^m$. 
Then it holds
\begin{equation*}
    \dim^*(S(U, \gamma)) = d - 1.
\end{equation*}
  \end{conjecture}
The conjecture suggests that the set \( S(U, \gamma) \) is finite if and only if \( U \) is irreducible. 
The last example shows that in general it is necessary to assume that the cone-volume vector is strictly positive. 

\begin{example}
    Consider the set $U = \{ e_1, -e_2,-e_1,e_2,e_1+e_2 \} \in \mathcal{U}(2,5)$ that is irreducible and the cone-volume vector $\gamma = \left(\frac{1}{3},\frac{1}{3},\frac{1}{9},\frac{1}{9},\frac{1}{9}\right)$. 
    The corresponding polynomial equations are
    \begin{align*}
        f_1(b) & = \left( - \sqrt{2} b_1 + b_2 + \sqrt{2} b_5 \right) \cdot \frac{b_1}{2} - \frac{1}{3}, \\
        f_2(b) & = \left( b_1 + b_3 \right) \cdot \frac{b_2}{2} - \frac{1}{3}, \\
        f_3(b) & = \left( b_2 + b_4 \right) \cdot \frac{b_3}{2} - \frac{1}{9}, \\
        f_4(b) & = \left( -\sqrt{2} b_4 + \sqrt{2} b_5 + b_3 \right) \cdot \frac{b_4}{2} - \frac{1}{9}, \\
        f_5(b) & = \left( - 2 b_5 + \sqrt{2} b_1 + \sqrt{2} b_4 \right) \cdot \frac{b_5}{2} - \frac{1}{9}, \\
        v(b) & = f_1(b) + f_2(b) + f_3(b) + f_4(b) + f_5(b).
    \end{align*}
    Using the software \emph{MomentPolynomialOpt.jl} \cite{MomentPolynomialOpt}, we can compute the solution set $S(U,\gamma) \approx \{ (0.66,0.82,0.15,0.66,0.79)^\intercal \}$, which is finite. 
    
    However, if we consider the cone-volume vector $\hat{\gamma} = (\frac{1}{4},\frac{1}{4},\frac{1}{4},\frac{1}{4},0)$ that is not strictly positive, we get $|S(U,\gamma)| = \infty$, since it corresponds to the reducible set $U' = \{ \pm e_1,\pm e_2 \}$. \hfill$\triangle$
\end{example}

\subsection*{Acknowledgement}

 This work is supported by the Deutsche Forschungsgemeinschaft (DFG, German Research Foundation) under Germany’s Excellence Strategy – The Berlin Mathematics Research Center MATH+ (EXC-2046/1, project ID: 390685689).

%\printbibliography

\end{document}